\documentclass[leqno,12pt]{article}

\usepackage{epsfig}

\usepackage{amsmath}
\usepackage{amscd}
\usepackage{amsopn}
\usepackage{amsthm}
\usepackage{amsfonts,amssymb}\usepackage{amsfonts,bbm}
\usepackage{latexsym}

\usepackage{color}

\makeatletter
\@addtoreset{equation}{section}
\makeatother

\newtheorem{lemma}{LEMMA}[section]
\newtheorem{proposition}[lemma]{PROPOSITION}

\newtheorem{theorem}[lemma]{THEOREM}

\newtheorem{remarks}[lemma]{REMARKS}

\newtheorem{assumption}[lemma]{ASSUMPTION}

\newcommand{\nat}{\mathbbm{N}}

\renewcommand{\a}{\alpha}
\renewcommand{\b}{\beta}

\newcommand{\g}{\gamma}

\newcommand{\ve}{\varepsilon}

\newcommand{\on}{\quad\text{ on }}
\newcommand{\und}{\quad\mbox{ and }\quad}
\newcommand{\inv}{^{-1}}
\newcommand{\ov}{\overline}

\newcommand{\W}{\mathcal W}  
\newcommand{\C}{\mathcal C}

\renewcommand{\H}{{\mathcal H}}
\newcommand{\B}{\mathcal B}

\newcommand{\M}{\mathcal M}

\newcommand{\itemframe}%
{\setlength{\parskip}{10pt}\begin{enumerate} \setlength{\topsep}{10pt}%
\setlength{\itemsep}{15pt}\setlength{\parsep}{5pt}}

\newcommand{\vx}{\ve_x}

\newcommand{\uc}{{U^c}}

\newcommand{\vc}{{V^c}}

\newcommand{\kap}{\operatorname*{cap}}

\newcommand{\kapi}{\operatorname*{{cap_\ast}}}
\newcommand{\kapo}{\operatorname*{{cap^\ast}}}

\newcommand{\es}{E_{\mathbbm P}}

 \title{H\"older continuity of harmonic functions for Hunt processes with Green  function
}
\author{Wolfhard Hansen} 
\date{}

\begin{document}
\maketitle

\begin{abstract}
Let $(X,\mathcal W)$ be a balayage space, $1\in \mathcal W$, or -- equivalently -- let $\mathcal W$
be the set of excessive functions of a Hunt process on a locally compact space~$X$ with countable base such that
$\mathcal W$  separates points, every function in $\mathcal W$ is the supremum of its
continuous minorants and there exist  strictly positive continuous       $u,v\in \mathcal W$ 
such that $u/v\to 0$  at infinity. We suppose that there is a~Green function 
$G>0$ for $X$, a metric $\rho$ on $X$ and a   decreasing function $g\colon[0,\infty)\to (0,\infty]$ 
having the doubling property and a mild upper decay  such that $G\approx g\circ\rho$
and the capacity of balls of radius $r$ is approximately $1/g(r)$.

It is shown that bounded harmonic functions are H\"older continuous, if the constant function $1$ is harmonic
and jumps out of balls admit a polynomial estimate. The latter is proven if scaling invariant Harnack inequalities hold.


 \end{abstract}

\section{Setting} 
Our basic setting will be  as in \cite{HN-harnack} to which we refer for further details
and   citations:
Let $X$ be a locally compact space with countable base.
Let $\C(X)$ denote the set of all continuous real functions on $X$
and let $\B(X)$ be the set of all Borel measurable numerical functions on $X$.
The set of all (positive) Radon measures on $X$ will be denoted by $\M(X)$.

Moreover, let $\W$ be a convex cone of positive lower 
semicontinuous numerical functions on~$X$ such that $1\in \W$ and $(X,\W)$ is a balayage space. 
In particular, the following holds: 
 
\begin{itemize} 
\item[\rm (C)]
 $\W$ separates the points of $X$, 
\[
              w=\sup\{v\in\W\cap \C(X)\colon v\le w\} \qquad \mbox{ for every } w\in \W,
\]
and there are strictly positive $u,v\in\W\cap \C(X)$ such that $u/v\to 0$  at~infinity.   
\end{itemize} 
There exists a Hunt process $\mathfrak X$ 
on $X$ such that $\W$ is the set $\es$ of excessive functions for the transition semigroup 
$\mathbbm P=(P_t)_{t>0}$ of $\mathfrak X$, 
that is,
\[
             \W=\{v\in \B^+(X)\colon \sup\nolimits_{t>0} P_tv=v\}.
\]
(Conversely, given any sub-Markov right-continuous semigroup $\mathbbm P=(P_t)_{t>0}$  on~$X$
such that (C) is satisfied by its convex cone $\es$ of excessive functions, $(X,\es)$ is a~balayage space, 
and $\mathbbm P$ is the transition semigroup of a Hunt process.)
 
For every subset $A$ of $X$, we have reduced functions  $R_u^A$, $u\in \W$, and reduced measures
$\vx^A$, $x\in X$, defined by
\[
          R_u^A:=\inf\{ v\in \W\colon v\ge u\mbox{ on } A\} \und  \int u\,d\vx^A=R_u^A(x).
\]
If $A$ is  a Borel set, then  
\begin{equation}\label{connection}
               R_1^A(x)= P^x[T_A<\infty], \qquad x\in X,
\end{equation} 
where  $T_A(\omega):=\inf \{t\ge 0\colon X_t(\omega)\in A\}$
and, for every Borel   set $B$ in
 $X$, 
\[
    \vx^A(B)=P^x[X_{T_A}\in B; T_A<\infty].
\]

For every open set $U$ in $X$, let $\H^+(U)$ denote the set of all 
functions $h\in \B^+(X)$ which are \emph{harmonic on~$U$} (in the sense of \cite{BH}), that is,
such that $h|_U\in \C(U) $  and
\begin{equation}\label{mv}
   \vx^\vc(h):=\int h\,d\vx^\vc=h(x) \qquad
\mbox{ if  }  V\mbox{ is open and } x\in V\subset
\subset U.
\end{equation} 
Analogously, we define the set $\H_b(U)$ of all bounded functions which are harmonic on $U$
and note that, given $h\in \B_b(X)$,  already (\ref{mv}) implies that 
$h|_U\in \C(U)$. 

We have the following  sheaf property: If $U_i$, $i\in I$, are open sets in $X$, then
\[ 
                             \bigcap\nolimits_{i\in I}  \H^+(U_i)= \H^+\bigl(\bigcup\nolimits_{i\in I} U_i\bigr).
\]
In fact, given an open set $U$ in $X$, a function $h\in \B^+(X)$ which is continuous on $U$ is already contained
in $\H^+(U)$, if, for every $x\in U$, there exists a fundamental system of relatively compact open neighborhoods $V$ of $x$ 
in $U$ such that $\vx^\vc(h)=h(x)$. 
        
Let us fix once and for all a point $x_0\in X$. In order to discuss H\"older continuity of bounded harmonic functions 
at $x_0$ we suppose the following (for an additional, later assumption see (\ref{jump-m})).

\begin{assumption}\label{main-ass} 
We have  a Borel measurable function $G\colon X\times X\to (0,\infty]$ with $G=\infty$
on the diagonal such that the following hold.
\begin{itemize} 
\item [\rm (i)] 
   For every $y\in X$,  $ G(\cdot,y)$ is a potential which is harmonic on $X\setminus \{y\}$.   
\item [\rm (ii)] 
For every potential $p$ on $X$, there exists $\mu\in \M(X)$ such that\\[-2mm]
\[
p=G\mu:=\int G(\cdot,y)\,d\mu(y).
\]
 \item [\rm (iii)]
There are constants $a_0\ge 0$ and  $c_1\ge 1$ such that, for all $a>a_0$,
\[     
             R_1^{\{G(\cdot,x_0)>a\}} \ge c_1\inv\,  \frac {G(\cdot, x_0)} a  \on \{G(\cdot,x_0)\le a\}.
\] 
\item[\rm (iv)] There is a~metric $\rho$
for $X$, a~decreasing function $g\colon [0,\infty)\to (0,\infty]$ and constants $c,c_D\in [1,\infty)$,          
$\a_0,\eta_0\in (0,1)$  such that, for every $r>0$, 
\[
g(r/2)\le c_D g(r),\qquad  g(r)\le \eta_0 g(\a_0 r) \und c\inv g\circ \rho\,  \le\,  G \, \le\,  c\, g\circ \rho.
\]
\end{itemize} 
\end{assumption}

 \begin{remarks}\label{g-small}
 {\rm 
1. Property (iv) implies that $g(0)=\lim_{r\to 0} g(r)=\infty$ and that, for \emph{any $\eta>0$}, there exists $\a\in (0,1)$ such that 
$g(r)\le \eta g(\a r)$ for every $r>0$ (choose $k\in\nat$ with $\eta_0^k\le \eta$ and take $\a:=\a_0^k$).

2. For 
applications leading to intrinsic H\"older continuity of bounded harmonic functions we recall the following 
(see \cite[Appendix]{H-liouville-wiener}).
Suppose that we have a function $G\colon X\times X\to (0,\infty]$ with  $G=\infty$
on the diagonal such that  (i) holds,  each potential $G(\cdot,y)$, $y\in X$, is bounded at infinity,
and $G$ has the triangle property, that is,  there exists $C>0$ such that
\[
         G(x,z)\wedge G(y,z)\le C G(x,y), \qquad x,y,z\in X.
\]
Then there exist  a~metric~$d$ for $X$ and $\g>0$ such that 
\[
                    G\approx d^{-\g},
\]
which clearly implies  (iv) with $g(r):=r^{-\g}$ and $d$ in place of $\rho$ (conversely, (iv)~implies that
$G$ has the triangle property).
  }
 \end{remarks}

Inner capacities for open sets $U$ in $X$ are defined by 
\begin{equation}\label{inner-cap}
 \kapi U:=\sup\bigl\{\|\mu\|\colon  \mu\in\M(X), \ \mu(X\setminus U)=0,\ G\mu\le 1\bigr\}
\end{equation} 
and outer capacities for arbitrary sets $A$ in $X$ by
\begin{equation}\label{outer-cap}
 \kapo A:=\inf\bigl\{ \kapi U\colon U\mbox{ open neighborhood of } A\bigr\}.
\end{equation} 
Obviously, $\kapo A=\kapi A$, if $A$ is open. If $\kapi A=\kapo A$, we may simply write $\kap A$
and speak of the capacity of $A$.

For   $r>0$, let
\[
           B(r):=\{x\in X\colon \rho(x,x_0)<r\},
\]
and let $R_0$ denote the supremum of all $r>0$ such that $B(r)$ is relatively compact 
and $cg(r)>a_0$. Then $0<R_0\le \infty$. Let
\[
          c_0:=c^3c_D c_1.  
\]

\begin{lemma}\label{cap-balls} 
 For all $0<r<R_0$,
\[
        c_0\inv g(r)\inv\le \kap B(r) \le c g(r)\inv.
\]
\end{lemma} 

\begin{proof} The second inequality is part of \cite[Proposition 1.7]{H-liouville-wiener} (and holds for all $r>0$).

To prove the first inequality we fix $0<r<R_0$ and note first that, by (iv), $ G(x,x_0)>c g(r)$ 
implies that $\rho(x,x_0)<r$,  and hence, by (iii),
\[
   R_1^{B(r)}\ge R_1^{\{G(\cdot,x_0)>cg(r)\}}\ge c_1\inv \,\frac{G(\cdot,x_0)}{cg(r)} 
\on \{G(\cdot,x_0)\le c g(r)\}\subset B(r)^c.
\]
By \cite[Proposition 1.10,(b)]{H-liouville-wiener}), this implies that
 $\kap B(r)\ge c_0\inv g(r)\inv$. 
\end{proof}

Let us note that conversely, by \cite[Proposition 1.10,(a)]{H-liouville-wiener}, 
any   estimate $\kap B(r)\ge C\inv  g(r)\inv$ implies that
$ R_1^{B(r)} \ge (c^2c_DC)\inv G(\cdot,x_0)/g(r)$ on $X\setminus B(r)$
(which in turn implies (iii)).

\begin{lemma}\label{shell-lemma}
Let $\b\in (0,1)$ such that $g(r)\le (2cc_0)\inv g(\b r)$ for all $r>0$.        
Then, for every    $0<r<R_0$, 
the shell $S(r):=B(r)\setminus B(\b r)$
satisfies $\kapo S(r)\ge (2c_0)\inv g(r)\inv$.
\end{lemma}

\begin{proof} By  Lemma \ref{cap-balls} and the subadditivity of $\kapo$, 
\[
         c_0\inv  g(r)\inv \le  \kap B(r)\le \kapo S(r) +\kap B(\b r),
\]
where $\kap B(\b r)\le c g(\b r)\inv \le (2c_0)\inv g(r)\inv$, by Lemma \ref{cap-balls}.
\end{proof}

\section{Control of jumps having Harnack inequalities} 

Let us  observe first that the probabilistic statements and proofs in this section can be replaced by analytic ones
using that, for all Borel   sets $A,B$ in an open set $U$ (where, as usual $\tau_U:= T_\uc$), 
\[
       P^x[X_{T_A}\in B; T_A<\tau_U] =\vx^{A\cup \uc}(B)
\]
(see \cite[VI.2.9]{BH}) and, for all Borel   sets $A,B$ in $X$ with $B\subset A $,
\begin{equation}\label{it-bal}
     \vx^B=\vx^A|_B+ \bigl(\vx^A|_{B^c}\bigr){}^B.
\end{equation} 
(If $x\in B$,  then (\ref{it-bal}) holds trivially. If $x\notin B$ and $u\in \W$,  then, by  \cite[VI.9.1]{BH},
\[
 R_u^B(x)= R_u^B(x)=\int R_u^B\,d\vx^A=\int_B u\,d\vx^A+\int_{B^c} \hat R_u^B\,d\vx^A.)
\]

Next we establish a useful estimate for the hitting of a union of two sets. 

\begin{lemma}\label{suff-hoelder}
Let $U,V$ be  open sets in $X$  such that the exit time $\tau_V$ is finite almost surely.
Let  $A$ be a~Borel set  in $U$, $B$ a~Borel set in $  V\setminus U$ and let $x\in U$ and   
$\kappa\ge 0$ such that $P^y[T_B\ge \tau_V] \ge \kappa  P^x[T_B\ge \tau_V]$ for every 
$y\in U\cap \ov A$. Then
\[
   P^x[T_{A\cup B}\ge\tau_V]\le  \bigl(1-\kappa P^x[T_A< \tau_U]\bigr)  P^x[T_B\ge \tau_V].
\]
\end{lemma}

\begin{proof} Defining $E:=[T_A\ge \tau_U]$ and $F:=[T_B\ge \tau_V]$ we have
\[
 P^x[T_{A\cup B} \ge \tau_V]=P^x[T_A\ge \tau_V, T_B\ge \tau_V]\le P^x(E\cap F)= P^x(F)- P^x(F\setminus E).
\]
 Clearly, 
$F \cap [\tau_V<\infty] =[X_{T_{B\cup\vc}}\in \vc]$ and 
$X_{T_{B\cup\vc}}=X_{T_{B\cup\vc}}\circ \theta_{T_A}$ on  $[T_A<\tau_U]$.
Since $X_{T_A}\in U\cap \ov A$ on $[T_A<\tau_U]$, the strong Markov property hence yields that 
\[
     P^x(F\setminus E)= \int_{[T_A<\tau_U]} P^{X_{T_A}}(F) \,dP^x
        \ge  \inf\nolimits_{y\in \ov A\cap U} P^y(F)\cdot P^x[T_A<\tau_U]. 
\]
By our assumption, the proof is completed combining the two estimates.
\end{proof}

As in \cite{HN-harnack}  let us define
\[
\eta:= (2c^3c_D^2)\inv.
\] 
By Remark \ref{g-small}, there exists $0<\a<1/2$ such that 
$g(r)\le \eta g(\a r)$ for every  $r>0$. Taking $\a_1:=\a/2$ we have
\begin{equation}\label{gar}
g((1-2\a_1) r)\le\eta g(\a_1 r) \qquad \mbox{ for every } r>0. \end{equation} 
Of course, (\ref{gar}) still holds if we replace
$\a_1$ by any $\a\in (0,\a_1)$.

Let us recall the following estimate  for the probability of hitting a set before leaving a large ball
(see    \cite[Proposition 3.2]{HN-harnack}). 

\begin{proposition}\label{hit-A}
For all   $r>0$, $0<\a\le \a_1$, 
$x\in  B(2\a r)$ and Borel  sets $A$ in~$B(2\a r)$,
\begin{equation}\label{hitting-est}
    P^x[T_A<\tau_{B(  r)} ]\ge \eta  g(\a r) \kapo(A).
\end{equation} 
\end{proposition} 

After these preparations we arrive at the main result of this section.

\begin{proposition}\label{harnack-jumps} 
Suppose that the constant function $1$ is harmonic and there exist   $0<\a\le \a_1$ and $K>0$ such that, 
for   $0<r<R_0$, 
\begin{equation}\label{harmonic-a}
  \sup h(B(\a r))\le K \inf h(B(\a r)) \quad\mbox{  for every } h\in \H_b^+(B(r)).
\end{equation} 
Let 
$a:=1-(2c_0K)\inv\eta$. 
Then, for all $0<r<R_0$ and $m=0,1,2,\dots$,  
\begin{equation}\label{jump-k}
           \vx^{B(\a^{2m}r)^c }(B(r)^c)\le a^m \qquad  \mbox{ for every } x\in B(\a^{2m} r).
\end{equation}
\end{proposition} 

\begin{proof} Of course, (\ref{jump-k}) holds trivially if $m=0$. Let us fix  $0<r<R_0$. For 
 $m=0,1,2,\dots$, we define
\[
B_m:=B(\a^m r) \und S_m:=B_0\setminus B_m.
\]
Then, for every $x\in B_m$,
\begin{equation}\label{hit-equiv}
           \vx^{B_m^c }(B_0^c)=P^x[T_{S_{2m}}\ge \tau_{B_0}].
\end{equation} 
For the moment, let us fix $m\in\nat$ and define
\[
A:=B_{2m-1}\setminus B_{2m},\quad U:=B_{2m-2}, \quad   V:=B_0, \quad 
B:=V\setminus U=S_{2m-2}.
\]  
By Proposition \ref{hit-A} and Lemma \ref{shell-lemma},   for all
$x\in B_{2m-1}$,  
\[
    P^x[T_A<\tau_U]\ge \eta g(\a^{2m-1}r) \kapo (A)\ge (2 c_0)\inv \eta.
\] 
The function $y\mapsto P^y[T_B\ge \tau_V]=\ve_y^{B\cup \vc}(\vc)$ is harmonic on $U$,
by \cite[VI.2.10]{BH}, and hence,
by (\ref{harmonic-a}),
\[
               P^y[T_B\ge \tau_V]\le K P^x[T_B\ge \tau_V]\qquad \mbox{ for all }y\in \ov A.
\]
So, by Lemma \ref{suff-hoelder}, for every $x\in B_{2m}$,
\[
          P^x[T_{S_{2m}}\ge \tau_V]\le     P^x[T_{A\cup B}\ge \tau_V]\le aP^x[T_B\ge \tau_V]
= a P^x[T_{S_{2m-2}}\ge \tau_V].
\]
In view of (\ref{hit-equiv}),  the proof of (\ref{jump-k}) is completed by induction.
\end{proof}

\section{H\"older continuity}

   In addition to Assumption \ref{main-ass} and   harmonicity of the constant function $1$, let us  suppose that there exist
   $a_0,\g\in (0,1)$ and $C_0\ge 1$ such that, for all $0<r<R_0$, $m\in\nat$ and $ x\in B(\g^{m} r)$,
\begin{equation}\label{jump-m}
           \vx^{B(\g^{m}r)^c }(B(r)^c)\le C_0 a_0^m.
\end{equation}

\begin{remarks}\label{suff-jump}{\rm
1.  By Proposition \ref{jump-k}, (\ref{jump-m}) holds if we have the 
Harnack inequalities (\ref{harmonic-a}). 
In \cite[Theorem 4.1]{kass-mim-jump}, H\"older continuity is obtained
assuming (more strongly) a version of Harnack inequalities for bounded  functions which are harmonic and positive 
on $B(r)$, but may be  negative on the complement.

2.  If $\mathfrak X$ is a diffusion, that is, if the reduced measures $\ve_x^{\uc}$
for open sets $U$ and $x\in U$ are supported by the  boundary of $U$, then (\ref{jump-m})
holds trivially.

3.  It is known that (\ref{jump-m}) holds for  many  L\'evy processes  (see \cite[Corollary 2]{grzywny}).   
}
\end{remarks}

To obtain a suitable H\"older exponent $\b$ we first 
 define  $\delta:=(12 c_0)\inv \eta$ and fix $b\in (1,\sqrt{3/2})$, $a\in (0,1/3)$  such that 
\begin{equation}\label{def-ab} 
b^2(1-3\delta)<1-2\delta \und ab^3(1-ab)\inv <\delta.
\end{equation} 
Then we choose $k\in\nat$ such that $a_0^k<C_0\inv a$, and define 
\begin{equation}\label{def-b}
  \a:=\a_1\wedge \g^k, \qquad\quad      \b:= (\ln b)\cdot (\ln \frac 1\a)\inv.
\end{equation} 
 By (\ref{jump-m}) and our choice of $\a$, for all $0<r<R_0$, $m=0,1,2,\dots $ and $x\in B(\a^m r)$,
\begin{equation}\label{jump-mn} 
                \vx^{B(\a^{m}r)^c }(B(r)^c)\le   a^m.
\end{equation}

\begin{theorem}\label{hoelder}
For all $0<r<R_0$, $        h\in \H_b(B(r))$ and $x\in B(r)$, 
\begin{equation}\label{hoelder-formula}
                     |h(x)-h(x_0)|\le 3 \|h\|_\infty \biggl(\frac{\rho(x,x_0)} {\a r}\biggr)^\b .
              \end{equation} 
\end{theorem} 

\begin{proof} Except for using  capacity instead of a certain measure, we may follow
rather closely the proof of \cite[Theorem 1.4]{kass-mim-intrinsic}.  

Let $0<r<R_0$  and $h\in\H_b(B)$,  $\|h\|_\infty=1$. 
For $n=0,1,2,\dots$  let 
\[
          B_n:= B(\a^n r), \quad m_n:=\inf h(\ov B_n), \quad M_n:=\sup h(\ov B_n).
\]
We claim that 
\begin{equation}\label{hoelder-ess}
            M_n-m_n\le s_n:=3 b^{-n}. 
\end{equation} 
Clearly, (\ref{hoelder-ess})  holds trivially for $n=0,1,2$,  since $M_n-m_n\le 2$ and $b^2<3/2$.
  Suppose that (\ref{hoelder-ess}) holds 
for some $n\in\nat$, $n\ge 2$. Given $\ve>0$, we may choose points $x,y\in B_{n+1}$ such that
\[
          h(x)>M_{n+1}-\ve \und   h(y) < m_{n+1}+\ve. 
\]
We intend to prove that
\[
            h(x)-h(y)\le   s_{n+1}. 
\]
To that end  we may assume without loss of generality that the closed set
\[
       A:= \biggl\{z\in \ov B_{n+1}\colon h(z)\le \frac {m_n+M_n}2\biggr\}
\]
satisfies $\kapo A\ge (1/2) \kap \ov B_{n+1}$. 
Indeed, otherwise we replace $h$ by $-h$ and exchange the roles of $x$ and $y$. 
Let
\[
           \mu:=\vx^{A\cup B_n^c}.
\]
Since $1, h\in \H_b(B_0)$, we know that $\mu$ is a probability measure and
\begin{equation}\label{mu-equation}
           h(x)-h(y)=\int  (h-h(y))\,d\mu.
\end{equation} 
The measure $\mu$ is supported by $  A\cup B_n^c$.
Clearly,
\[
    \int_A (h-h(y))\,d\mu\le \bigl(\frac{m_n+M_n}2-m_n\bigr) \mu(A)\le\frac 12 s_n \mu(A)\le
\frac 12 s_{n-1} \mu(A),
\]
where, by  Proposition \ref{hit-A} and Lemma \ref{cap-balls},
\[
             \mu(A) \ge \eta g(\a^{n+1} r) \kapo A \ge \eta g(\a^{n+1} r) \cdot  (2 c_0)\inv g(\a^{n+1} r)\inv =6 \delta. 
\]
Since $\mu(B_n^c)=1-\mu(A)$, we have
\[
           \int_{B_{n-1}\setminus B_n} (h-h(y)) \,d\mu \le s_{n-1} (1-\mu(A)).
\]
Combining the three estimates we obtain that
\begin{equation}\label{mu-ess-2}
\int_{A\cup (B_{n-1}\setminus B_n)} (h-h(y)\,d\mu \le s_{n-1} (1-\frac 12 \mu(A)) \le s_{n+1} b^{2} (1-3\delta) .
\end{equation} 
Finally, 
\begin{equation}\label{mu-ess-3}
          \int_{B_{n-1}^c} (h-h(y))\,d\mu\le  2\mu(B_0^c) +\sum\nolimits_{j=0}^{n-2} s_j  \mu(B_j \setminus B_{j+1}). 
\end{equation} 
By (\ref{it-bal}) and (\ref{jump-mn}),  for every $1\le m \le n$,
\[ 
         \mu(B_m^c)\le \vx^{B_n^c}(B_m^c)\le a^{n-m}.
\]
So, by (\ref{def-ab}),  $2\mu(B_0^c)\le 2a^n\le s_{n+1} \delta$   (note that $ab<\delta$) 
and 
\[
      \sum\nolimits _{j=0}^{n-2} s_j \mu(B_j\setminus B_{j+1}) \le \sum\nolimits _{j=0}^{n-2} s_j \mu( B_{j+1}^c) 
\le     3  \sum\nolimits _{j=0}^{n-2} b^{-j} a^{n-(j+1)}
                =  s_{n+1} s, 
\]
where, 
\[     s =b^{n+1} \sum\nolimits_{j=0}^{n-2} b^{-j} a^{n-(j+1)}=b^2 \sum\nolimits_{j=0}^{n-2}   (ab)^{n-(j+1)}
\le   \frac{a b^{3}} { 1-ab} \le \delta.
\]
Having (\ref{mu-equation}), the estimates  (\ref{mu-ess-2}) and (\ref{mu-ess-3})
hence yield  that $h(x)-h(y)\le s_{n+1}$.
Thus $M_{n+1}-m_{n+1}\le s_{n+1}$, since $\ve>0$ was arbitrary, and the inductive step for (\ref{hoelder-ess})
 is finished.

 Given $x\in B_0\setminus \{x_0\}$, there exists $n\ge 0$ 
such that $x\in B_n\setminus B_{n+1}$, and therefore, by~(\ref{def-b}), 
\[
               |h(x)-h(x_0)|\le  3b^{-n} = 3\a^{ n\b} \le 3 \biggl(\frac{\rho(x,x_0)}{\a r}\biggr)^\b
\]
completing the proof.
\end{proof}

\bibliographystyle{plain} 
\def\cprime{$'$} \def\cprime{$'$}

{\small \noindent 
Wolfhard Hansen,
Fakult\"at f\"ur Mathematik,
Universit\"at Bielefeld,
33501 Bielefeld, Germany, e-mail:
 hansen$@$math.uni-bielefeld.de}

\end{document}